\documentclass[10pt]{article}

\usepackage[top=1.25in, bottom=1.25in, left=1.25in, right=1.25in]{geometry}
\usepackage{amsmath,amsthm,amssymb,xurl}
\usepackage[hidelinks]{hyperref}
\usepackage{tikz,authblk}
\usepackage{textcomp} 
\usepackage{fancybox}
\usepackage{float}
\usepackage{longtable}
\usepackage{authblk}
\usepackage{xcolor}

\newcounter{algcounter}
\renewcommand{\thealgcounter}{\arabic{algcounter}}

\newenvironment{algorithm}[2]%
{\refstepcounter{algcounter}%
\label{#1}%
\begin{Sbox}\begin{minipage}{\boxwidth}%
\mbox{\quad}\vspace{-.025in}\\ \af{Algorithm~\thealgcounter : #2} 
\noindent}%
 {\mbox{\quad}\vspace{-.2in} \end{minipage}
\end{Sbox}\bigskip\noindent\fbox{\TheSbox}
}

\usetikzlibrary{shapes.geometric}
\newcommand{\zed}{{\ensuremath{\mathbb{Z}}}}
\newcommand{\eff}{{\ensuremath{\mathbb{F}}}}

\newcommand{\Aut}{{\ensuremath{\mathsf{Aut}}}}

\newcommand{\af}[1]{\textbf{#1}} 

\newlength{\boxwidth}
\newtheorem{theorem}{Theorem}[section]

\newtheorem{lemma}[theorem]{Lemma}
\newtheorem{corollary}[theorem]{Corollary}
\newtheorem{example}{Example}[section]
\newtheorem{remark}{Remark}[section]
\newtheorem{definition}{Definition}[section]

\floatstyle{plain}
\newfloat{flt}{htbp}{.flt}[section]

\allowdisplaybreaks

\title{Uniqueness and explicit computation of mates in near-factorizations
}
\author[1]{Donald L.\ Kreher}
\affil[1]{Department of Mathematical Sciences,
Michigan Technological University, Houghton, MI 49931, USA, {\tt kreher@mtu.edu}}
\author[2]{William J.\ Martin\thanks{W.J.\ Martin’s research is supported by NSF DMS Award \#1808376.}}
\affil[2]{Department of Mathematical Sciences, Worcester Polytechnic Institute, Worcester MA, 01609, USA, {\tt martin@wpi.edu}}
\author[3]{Douglas R.\ Stinson\thanks{D.R.\ Stinson's research is supported by  NSERC discovery grant RGPIN-03882.}}
\affil[3]{David R.\ Cheriton School of Computer Science, University of Waterloo, Waterloo ON, N2L 3G1, Canada, {\tt dstinson@uwaterloo.ca}}

\begin{document}
\maketitle

\begin{abstract}
We 
show that a ``mate'' $B$ of a set $A$ in a near-factorization $(A,B)$ of a finite group $G$ is unique. Further, we  describe how to compute the mate $B$ very efficiently using an explicit formula for $B$. We use this approach to give an alternate proof of a theorem of Wu, Yang and Feng, 
which states that a strong circular external difference family cannot have more than two sets.
We prove some new structural properties of near-factorizations in certain classes of groups. Then we examine all the noncyclic abelian groups of order less than $200$ in a search for a possible nontrivial near-factorization.  All of these possibilities are ruled out, either by theoretical criteria or by exhaustive computer searches. (In contrast, near-factorizations in cyclic or dihedral groups are known to exist by previous results.) We also look briefly at  nontrivial near-factorizations of index $\lambda > 1$ in 
noncyclic abelian groups. Various examples are found with $\lambda = 2$ by computer.
\end{abstract}

\medskip


\medskip


\section{Mates in near-factorizations}
\label{sec1}

We begin with a definition.

\begin{definition}[near-factorization]
Let $(G,\cdot)$ be a finite multiplicative group with identity $e$. For $A, B \subseteq G$, define 
\[AB = \{ gh \colon g \in A, h \in B\}.\] We say that $(A,B)$ is a \emph{near-factorization} of $G$
if $|A| \times |B| = |G|-1$ and $G \setminus \{e\} =AB$. \textup{(}In the case where we have an additive group $(G,+)$ with identity $0$, the second condition becomes $G \setminus \{0\} =A+B$.\textup{)} Further, $(A,B)$ 
is an \emph{$(r,s)$-near-factorization} of $G$ if $|A|=r$ and $|B| =s$, where 
$rs = |G| -1$. 
\end{definition}

There is always a \emph{trivial} $(1,|G|-1)$-near-factorization of $G$ given by $A = \{e\}$, $B = G \setminus \{e\}$. 
A near-factorization with $|A| > 1$ and $|B| > 1$ is \emph{nontrivial}.
If $(A,B)$ is a near-factorization, then we say that $B$ is a \emph{mate} of $A$.

Near-factorizations and related objects have been studied in numerous papers, including \cite{BHS,DB,CGHK,Grin,KPS,Pech03,Pech,SS,Shin}.
Near-factorizations in cyclic or dihedral groups are known to exist; however, there is no known nontrivial near-factorization in any noncyclic abelian group. 

We discuss a few basic properties of near-factorizations now. If $G$ is an abelian group, it is clear that 
$AB = BA$ for any two subsets $A$ and $B$ of $G$. Hence $(A,B)$ is a near-factorization of $G$ if and only if 
$(B,A)$ is a near-factorization of $G$. It follows that there is an $(r,s)$-near-factorization of an abelian group $G$ if and only if there is an $(s,r)$-near-factorization of $G$. 

In this paper, we focus on abelian groups. Nevertheless, it is interesting to consider if the above result also holds for nonabelian groups. We can prove a similar result, based on the identity $(gh)^{-1} = h^{-1}g^{-1}$.
Because $AB = G \setminus \{e\}$ if and only if $(AB)^{-1} = G \setminus \{e\}$, we see that 
$AB = G \setminus \{e\}$ if and only if $B^{-1}A^{-1} = G \setminus \{e\}$. Hence, $(A,B)$ is a near-factorization of $G$ if and only if 
$(B^{-1},A^{-1})$ is a near-factorization of $G$. Because $(B^{-1},A^{-1})$ is an $(s,r)$-near-factorization of $G$, we have the following result.

\begin{lemma} 
Suppose $G$ is any finite group. Then there is an $(r,s)$-near-factorization of $G$ if and only if there is an $(s,r)$-near-factorization of $G$. 
\end{lemma}

A subset $A \subseteq G$ is \emph{symmetric} if  $g \in A$ if and only if $g^{-1} \in A$.
A near-factorization $(A,B)$ of a group $G$ is symmetric if $A$ and $B$ are both symmetric.
The following important result from de Caen \emph{et al.} \cite{CGHK} will be used throughout this paper. 

\begin{lemma}
[D. de Caen {\it et al.}, {\cite[Proposition 2]{CGHK}}]
\label{symmetric.lem}
Suppose $G$ is an abelian group. If there is an $(r,s)$-near-factorization of $G$, then there is a
symmetric $(r,s)$-near-factorization of $G$.
\end{lemma}


Suppose $|G| = n$ and the elements of $G$ are enumerated as
$g_1, \dots , g_n$. For convenience, we will assume that $g_1 = e$ (the identity).
For any subset $H \subseteq G$, define an $n \times n$ matrix $M(H)$ as follows:
\begin{equation}
\label{M.eq}
M(H)_{i,j} = \begin{cases}
1 & \text{if $(g_i)^{-1}g_j \in H$}\\
0 & \text{otherwise}.
\end{cases}
\end{equation}

Let $J$ denote the all-ones matrix and let $I$ be the identity matrix.
The following important theorem holds for all finite groups (abelian and nonabelian).

\begin{theorem}[D. de Caen {\it et al.}, {\cite{CGHK}}]
\label{equiv.thm}
Suppose $A,B \subseteq G$. Then $(A,B)$ is a near-factorization of $G$ if and only if
$M(A)M(B) = J-I$ \textup{(}this is an equation over the integers\textup{)}. 
\end{theorem}

\begin{theorem}
\label{T2}
Let $n$ be a positive integer and let $J - I = XY$ be a factorization into $n \times n$ integral matrices $X$ and $Y$. Suppose also that $XJ = rJ$. Then
\begin{equation}
\label{eq2}
 Y = \frac{1}{r} J - X^{-1}.
 \end{equation}
\end{theorem}

\begin{proof}
Because $\det(J-I) \neq 0$ and $J - I = XY$, it follows that $X$ is invertible.
Then we have the following:
\begin{align*}
Y &= X^{-1}(J-I)\\
&=  X^{-1}J - X^{-1} \\
&= \frac{1}{r}X^{-1}(rJ)-X^{-1}\\
&=\frac{1}{r}X^{-1}(XJ)-X^{-1} \\
&=\frac{1}{r}J-X^{-1}.
\end{align*}
\end{proof}

\begin{remark}
{\rm Theorem \ref{T2} also follows from the proof of \cite[Theorem 2]{CGHK}. The proof we give is a bit more direct.}
\end{remark}

Hence, we immediately obtain the following result as a corollary of Theorem \ref{equiv.thm}. Again, this result holds for all finite groups.

\begin{theorem}
\label{T3}
Suppose $(A,B)$ and $(A,B')$ are both near-factorizations of a finite group $G$. Then $B = B'$.
\end{theorem}
Informally, if $A$ has a mate $B$, then the mate is unique. 

\medskip

The remainder of this paper is organized as follows. In Section \ref{algs.sec}, we discuss two versions of an algorithm to compute mates. Section \ref{structural.sec} discusses some interesting structural properties of near-factorizations. Section \ref{nonexist.sec} reviews known existence results for near-factorizations. It was previously known (see \cite{CGHK,Pech}) that there is no nontrivial near-factorization of a noncyclic abelian group of order less than $100$. 
By employing a variety of theoretical techniques along with exhaustive computer searches, we extend this result, showing that 
there is no nontrivial near-factorization of a noncyclic abelian group of order less than $200$. Section \ref{SCEDF.sec} gives a simple new proof of the nonexistence of strong circular external difference families having more than two sets. In Section \ref{lambda.sec}, we find examples of nontrivial near-factorizations of  noncyclic abelian groups having index $\lambda = 2$.
Finally Section \ref{summary.sec} is a brief summary.

\subsection{Algorithms to compute mates}
\label{algs.sec}

Theorems \ref{T2} and \ref{T3} provide a convenient method to search for $(r,s)$-near-factorizations in a specified finite group $G$ of order $n$. 
We can consider all possible sets $A \subseteq G$ with $|A| = r$, and for each $A$, determine if there is a mate $B$. If $G$ is an abelian group, then we can assume that $A$ is {symmetric}, from Lemma \ref{symmetric.lem}.

We also need to recognize when a given $A$ does not have a mate.
Suppose we are given a set $A \subseteq G$ with $|A| = r$, where $|G|=n$ and $r$ divides $n-1$.
We would first construct the matrix $X= M(A)$. Then, if $X$ is invertible, $Y$ can be computed such that $XY = J - I$. Then the mate $B$ can be obtained from $Y$. We can proceed as shown in Algorithm \ref{alg1}.



\begin{flt}
\begin{algorithm}{alg1}{Computing a mate}
\begin{enumerate}
\item Input: A finite group $G$ and a subset $A \subseteq G$ with $|A| = r$, where $|G| \equiv 1 \bmod r$.
\item Let $s = (n-1)/r$ (if there is a near-factorization $(A,B)$ of $G$, then $|B| = s$). Compute $X = M(A)$. From (\ref{M.eq}), it is easily seen that the first row of $X$ has a $1$ in column $g$ if and only if $g \in A$. $X$ has row and column sums equal to $r$, so $XJ = rJ$.
\item Compute $X^{-1}$ over the rationals.  If $X^{-1}$ does not exist, then QUIT (in this case, $A$ does not have a mate $B$). 
$X^{-1}$ 
has all row and column sums equal to $1/r$.
\item Compute \[ Y  = \frac{1}{r} J - X^{-1}.\] 
$Y$ is a 
matrix with row and column sums equal to 
$ \frac{n}{r} - \frac{1}{r} = s$.
If $Y$ is not a $0$-$1$ matrix, then QUIT ($A$ does not have a mate $B$ in this case).
\item Compute $B = \{g \colon  \text{the first row of $Y$ has a $1$ in column $g$}\}.$ Note that it must be the case that
$|B| = s$.
The pair of sets $(A,B)$ is the unique near-factorization of $G$ in which $A$ is one of the two sets.
\end{enumerate}
\end{algorithm}
\end{flt}

We observe that, if $G$ is the cyclic group $\zed_n$, then the matrices $X$, $X^{-1}$ and $Y$ are circulant matrices.

\begin{example}
\label{exam1}
{\rm
Suppose $G = \zed_7$, so $n = 7$, and suppose $A = \{0,3\}$. Because we are working in an additive group, we seek $B$ such that $|B| = 3$ and
$A+B = \zed_7 \setminus \{0\}$. In step 1, we obtain
\[ X = M(A) = 
\left(
\begin{array}{ccccccc}
1 & 0 & 0 & 1 & 0 & 0 & 0 \\
0 & 1 & 0 & 0 & 1 & 0 & 0 \\
0 & 0 & 1 & 0 & 0 & 1 & 0 \\
0 & 0 & 0 & 1 & 0 & 0 & 1 \\
1 & 0 & 0 & 0 & 1 & 0 & 0 \\
0 & 1 & 0 & 0 & 0 & 1 & 0 \\
0 & 0 & 1 & 0 & 0 & 0 & 1 \\
\end{array}
\right).
\]
The rows and columns of $X$ are labelled (from left to right and top to bottom) as $0, 1, \dots , 6$.

In step 2, we compute 
\[ X^{-1} = 
\left(
\begin{array}{rrrrrrr}
\frac{1}{2} & -\frac{1}{2} & -\frac{1}{2} & -\frac{1}{2} & \frac{1}{2} & \frac{1}{2} & \frac{1}{2} \vspace{.05in}\\
\frac{1}{2} & \frac{1}{2} & -\frac{1}{2} & -\frac{1}{2} & -\frac{1}{2} & \frac{1}{2} & \frac{1}{2} \vspace{.05in}\\
\frac{1}{2} & \frac{1}{2} & \frac{1}{2} & -\frac{1}{2} & -\frac{1}{2} & -\frac{1}{2} & \frac{1}{2} \vspace{.05in}\\
\frac{1}{2} & \frac{1}{2} & \frac{1}{2} & \frac{1}{2} & -\frac{1}{2} & -\frac{1}{2} & -\frac{1}{2} \vspace{.05in}\\
-\frac{1}{2} & \frac{1}{2} & \frac{1}{2} & \frac{1}{2} & \frac{1}{2} & -\frac{1}{2} & -\frac{1}{2} \vspace{.05in}\\
-\frac{1}{2} & -\frac{1}{2} & \frac{1}{2} & \frac{1}{2} & \frac{1}{2} & \frac{1}{2} & -\frac{1}{2} \vspace{.05in}\\
-\frac{1}{2} & -\frac{1}{2} & -\frac{1}{2} & \frac{1}{2} & \frac{1}{2} & \frac{1}{2} & \frac{1}{2} \vspace{.05in}\\
\end{array}
\right).
\]

Then, in step 3, we compute 
\begin{align*}
 Y &= M(B)\\
 &= \frac{1}{2}J - X^{-1}\\
&=
\left(
\begin{array}{ccccccc}
0 & 1 & 1 & 1 & 0 & 0 & 0 \\
0 & 0 & 1 & 1 & 1 & 0 & 0 \\
0 & 0 & 0 & 1 & 1 & 1 & 0 \\
0 & 0 & 0 & 0 & 1 & 1 & 1 \\
1 & 0 & 0 & 0 & 0 & 1 & 1 \\
1 & 1 & 0 & 0 & 0 & 0 & 1 \\
1 & 1 & 1 & 0 & 0 & 0 & 0 \\
\end{array}
\right).
\end{align*}
The set $B$ is given by the positions of the $1$'s in the first row of $Y$: $B = \{1,2,3\}$. Thus the unique mate for $\{0,3\}$ is
$\{1,2,3\}$. We can easily verify that $\{0,3\} + \{1,2,3\} = \{1,2,3,4,5,6\}$.
}$\hfill\blacksquare$
\end{example}

Step 2 in Algorithm \ref{alg1} requires finding an inverse of an $n \times n$ matrix. This can be replaced by solving a sparse system of $n$ linear equations in $n$ unknowns. We present an optimized algorithm now, which  
calculates $B$ from the solution to this linear system.

Consider the matrix equation $XY = J-I$. We really only need the first column of the matrix $Y$.
If $X$ is invertible, then it suffices to solve the linear system 
\begin{equation}
\label{system.eq}
 X (z_0, \dots , z_{n-1})^T = (0,1, \dots , 1)^T
 \end{equation}
to obtain the vector $\mathbf{z}  = (z_0, \dots , z_{n-1})$.
This linear system consists of $n$ equations in $n$ unknowns. 
Suppose $\mathbf{z}$ is a $0$-$1$ vector. Then it must have Hamming weight equal to $s$ and we can compute $B$ from $\mathbf{z}$.
Because $\mathbf{z}^T$ is the first column of $Y$, it follows from (\ref{M.eq}) that
\[z_j = y_{j,0} = 1 \Leftrightarrow -j \in B.\]
Hence, the set $B$ can be  computed as
\begin{equation}
\label{B.eq}
B = \{ -g  \colon  z_{g} = 1\}.
\end{equation} 

Our modified algorithm is presented as Algorithm \ref{alg2}.

\begin{flt}
\begin{algorithm}{alg2}{Computing a mate (optimized version)}
\begin{enumerate}
\item 
Solve the linear system (\ref{system.eq}). (If the system does not have a unique solution, then QUIT.)
\item If $\mathbf{z}$ is not a $0$-$1$ vector, then  QUIT ($A$ does not have a mate in this case).
\item Compute $B$ using (\ref{B.eq}). The pair of sets $(A,B)$ is the unique near-factorization of $G$ in which $A$ is one of the two sets.
\end{enumerate}
\end{algorithm}
\end{flt}



\begin{example}
{\rm
As in Example \ref{exam1}, let $G = \zed_7$ and $A = \{0,3\}$.
The linear system we want to solve is 
\begin{align*}
z_0 + z_3 & = 0\\
z_1 + z_4 & = 1\\
z_2 + z_5 & = 1\\
z_3 + z_6 & = 1\\
z_4 + z_0 & = 1\\
z_5 + z_1 & = 1\\
z_6 + z_2 & = 1.
\end{align*}
The unique solution to this system is 
\[ z_4 = z_5 = z_6 = 1, \quad z_0 = z_1  = z_2 = z_3 = 0.
\]
Hence, using (\ref{B.eq}), we get 
$B = \{ 1,2,3\}$, agreeing with  Example \ref{exam1}.
}$\hfill\blacksquare$
\end{example}

The system (\ref{system.eq}) of $n$ equations in $n$ unknowns is sparse, because each equation only involves $|A|$ variables. We can assume $|A| < |B|$, so $|A| < \sqrt{n}$, because $|A| \times |B| = n-1$. Hence, there are at most $\sqrt{n}$ nonzero coefficients in each linear equation.


Both algorithms for computing a mate $B$ for a given $A$ are very fast.
As an example, suppose $G = \zed_{199}$ and 
$A = \{0, 1, 2, 3, 4, 195, 196, 197, 198\}$. The matrix $X$ is a $199 \times 199$ matrix. Algorithm \ref{alg1} explicitly computes the matrix $X^{-1}$. Here, the \textbf{MatrixInverse} command in Maple takes about 3.2 seconds to execute.\footnote{We are running Maple 2024 on a 32 GB 2021 MacBook Pro with an Apple M1 Max chip.} 

On the other hand, in Algorithm \ref{alg2}, we solve the  sparse linear system (\ref{system.eq}). The relevant Maple command to do this is \textbf{LinearSolve} with
the \texttt{method=\textquotesingle modular\textquotesingle} option.
This approach requires less than $.14$ seconds for the same example in $\zed_{199}$, so it is more than 20 times faster.

\medskip

In previous papers, searches for near-factorizations have involved first choosing $A$ and then doing an exhaustive search for a mate $B$. 
This basically requires examining all possible sets $B$ of the appropriate size to see if one of them is a mate to the given $A$.
The approach we have described replaces the exhaustive search for $B$ by a simple direct computation, as described in Algorithm \ref{alg1} or \ref{alg2}.

\section{Structural properties of near-factorizations}
\label{structural.sec}

\subsection{Involutions}
\label{involutions.sec}

Suppose we want to perform an exhaustive search for an $(r,s)$-near-factorization of an abelian group $G$ of order $n$. Basically, we will generate all symmetric subsets $A \subseteq G$ with $|A| = r$. Then, for each $A$, we run Algorithm \ref{alg2} to see if a mate $B$ exists. 

We define an \emph{involution} in an additive group $G$ to be any element $x$ such that $x = -x$. Thus we are treating $0$ (the identity) as an involution in this paper.
Because $A$ is symmetric, it will consist of involutions along with pairs of group elements $\{x, -x\}$ where $x$ is not an involution (we say that these are \emph{symmetric pairs}). Suppose that $G$ contains $t_1$ involutions; then there are $t_2 = (n-t_1)/2$ symmetric pairs in $G$.  $A$ will consist of $i_1$ involutions and
$i_2 = (r-i_1)/2$ symmetric pairs, where $0 \leq i_1 \leq t_1$ and $r-i_1$ is even. For each possible value of $i_1$, we need to consider all sets $A$ that consist of $i_1$ of the $t_1$ possible involutions along with $i_2$ of the $t_2$ possible symmetric pairs. 
However, for certain groups $G$, we can speed up the search by limiting the number of sets $A$ that we need to examine. We can reduce the size of the search by considering automorphisms of $G$. 

To illustrate the main ideas, suppose $G = \zed_t \times (\zed_2)^k$, where $t$ is odd. 
We will assume $G$ is non-cyclic, so $k \geq 2$. 
Because $t$ is odd, we have $\Aut(G) = \Aut(\zed_t) \times \Aut((\zed_2)^k)$. 
The automorphism group  $\Aut(\zed_t)$ consists of all mappings of the form $x \mapsto cx$, where $\gcd(c,x) = 1$. The automorphism group $\Aut((\zed_2)^k)$ is the general linear group $\mathrm{GL}(k,2)$, which consists of of all $k$ by $k$ invertible matrices over $\zed_2$.

Because $|G|$ is even, $|A|$ is odd and therefore we require an odd number of involutions in $A$. 
The involutions in $G$ have the form $(0, x)$, where $x$ is any element in $(\zed_2)^k$. 
Suppose we only want one involution in $A$, say $(0,x)$. Then there are two possible cases up to equivalence: $x = 0$ and $x \neq 0$ (note all $2^k-1$ values $x \neq 0$ are equivalent under a linear transformation in $\mathrm{GL}(k,2)$).

Suppose we want $i_1$ involutions in $A$. Then we need to find all the orbits of $i_1$-subsets of involutions under $\mathrm{GL}(k,2)$. In the case $i_1= 1$, there are two orbits, as we discussed above. 

If $k = 2$ and $i_1 = 3$, there are four possible $3$-subsets of elements of $(\zed_2)^2$.
These four subsets are partitioned into two orbits under the action of $\mathrm{GL}(k,2)$.
Orbit representatives are as follows:
\begin{enumerate}
\item $\bigl\{(0,0), (0,1), (1,0)\bigr\}$; and
\item $\bigl\{(0,1), (1,0), (1,1)\bigr\}$.
\end{enumerate}

The next case is $k = 3$ and $i_1 = 3$. There are $\binom{8}{3} = 56$ $3$-subsets of elements of $(\zed_2)^3$, which are partitioned into three orbits. Orbit representatives as follows:
\begin{enumerate}
\item $\bigl\{(0,0,0), (0,0,1), (0,1,0)\bigr\}$ (the $0$-vector and two other linearly independent vectors);
\item $\bigl\{(0,0,1), (0,1,0), (0,0,1)\bigr\}$ (three linearly independent vectors); and
\item $\bigl\{(0,0,1), (0,1,0), (0,1,1)\bigr\}$ (two linearly independent vectors and their sum).
\end{enumerate}
Here, there is a considerable savings. Instead of testing all $\binom{8}{3} = 56$ sets of three involutions, we just need to consider three sets of three involutions (one from each orbit).

Finally, when we choose the symmetric pairs to be included in $A$, we can assume that the first pair chosen is a minimum representative of an orbit under the action of  $\Aut(\zed_t)$. If $t$ is prime, then this minimum representative has first co-ordinate equal to $1$.

\medskip

In Section \ref{nonexist.sec}, we will also consider some groups of the form 
$G = \zed_t \times \zed_2 \times \zed_4$, where $t$ is odd. 
We have $\Aut(G) = \Aut(\zed_t) \times \Aut(\zed_2 \times \zed_4) $ and we
focus on $\Aut(\zed_2 \times \zed_4)$.
The group $G$ has three non-identity involutions,
namely $(0,1,0)$, $(0,1,2)$ and $(0,0,2)$. 
Using techniques discussed in \cite{HR}, the following facts can be established. First, it can be shown that
$\Aut(\zed_2 \times \zed_4 ) = \zed_2 \times \zed_4$. Every  automorphism in $\Aut(\zed_2 \times \zed_4 )$ fixes $(0,0)$ and $(0,2)$. Finally $(1,0)$, $(1,2)$ are in the same orbit under the action of $\Aut(\zed_2 \times \zed_4)$.
It follows that, if $A$ contains exactly one involution, say $x$, then an exhaustive search will need to consider three possible cases: $x = (0,0,0)$, $(0,1,0)$ and $(0,0,2)$.

We fill in a few details to justify these assertions. A mapping $\theta \in \Aut(\zed_2 \times \zed_4 )$ is determined by $\theta(1,0)$ and $\theta(0,1)$. First, $\theta(0,1)$ must be an element of order four, 
so \[\theta(0,1) \in X = \{ (0,1), (0,3), (1,1), (1,3)\} .\]
Next,  $\theta(1,0)$ must be an element of order two, 
so \[\theta(1,0) \in \{ (1,0), (0,2), (1,2)\}.\] However, it is easy to see that 
$(0,2) \in \langle x \rangle$ for all $x \in X$. Hence, $\theta(1,0) = (0,2)$ is not possible and therefore
\[\theta(1,0) \in \{ (1,0), (1,2)\}.\]
Thus we have eight automorphisms. It is now a simple computation to verify that 
each of these eight automorphisms fixes $(0,2)$. 
Further, four automorphisms fix $(1,0)$ and $(1,2)$, and the other four 
automorphisms interchange $(1,0)$ and $(1,2)$.

\medskip

We obtain a  similar result for the group $G = \zed_t \times \zed_2 \times \zed_8$ (where $t$ is odd). This group $G$ has three non-identity involutions,
namely $(0,1,0)$, $(0,1,4)$ and $(0,0,4)$. Each automorphism in $\Aut(\zed_2 \times \zed_8 )$ fixes $(0,0)$ and $(0,4)$; and  $(1,0)$ and $(1,4)$ are in the same orbit under the action of $\Aut(\zed_2 \times \zed_8)$.
Hence, if $A$ contains exactly one involution, say $x$, then an exhaustive search will need to consider three possible cases: $x = (0,0,0)$, $(0,1,0)$ and $(0,0,4)$.


\subsection{Groups of the form $G = \zed_t \times (\zed_2)^2$}

We now establish some additional structural properties of near-factorizations in $G = \zed_t \times (\zed_2)^2$, where $t$ is odd.  Suppose $(A,B)$ is an $(r,s)$-near-factorization of $G$. Because $|G| = 4t \equiv 0 \bmod 4$, it follows that one of $r$ and $s$ is congruent to $1$ modulo $4$ and the other is congruent to $3$ modulo $4$. By interchanging $A$ and $B$ if necessary, we can assume that
  $r \equiv 1 \bmod 4$ and  $s \equiv 3 \bmod 4$.

For $i,j \in \{0,1\}$, define $G_{i,j} = \zed_t \times \{i\} \times \{j\}$, let $A_{i,j} = A \cap G_{i,j}$ and  let $B_{i,j} = B \cap G_{i,j}$. 
Denote $a_{i,j} = |A_{i,j}|$ and denote $b_{i,j} = |B_{i,j}|$.
We can write down several equations involving the $a_{i,j}$'s and $b_{i,j}$'s:
\begin{align}
a_{0,0}b_{0,0} + a_{0,1}b_{0,1} + a_{1,0}b_{1,0} + a_{1,1}b_{1,1} &= t-1\label{eq11}\\
a_{0,0}b_{0,1} + a_{0,1}b_{0,0} + a_{1,0}b_{1,1} + a_{1,1}b_{1,0} &= t\label{eq12}\\
a_{0,0}b_{1,0} + a_{0,1}b_{1,1} + a_{1,0}b_{0,0} + a_{1,1}b_{0,1} &= t\label{eq13}\\
a_{0,0}b_{1,1} + a_{0,1}b_{1,0} + a_{1,0}b_{0,1} + a_{1,1}b_{0,0} &= t\label{eq14}\\
a_{0,0} + a_{0,1} + a_{1,0} + a_{1,1} &= r\label{eq15}\\
b_{0,0} + b_{0,1} + b_{1,0} + b_{1,1} &= s.\label{eq16}
\end{align}
Equation (\ref{eq11}) is obtained by counting sums  $g+h \in G_{0,0}$, where $g \in A$, $h \in B$. Similarly, equations (\ref{eq12})--(\ref{eq14}) result from counting sums in $G_{0,1}$, $G_{1,0}$ and $G_{1,1}$, respectively. Equations (\ref{eq15}) and (\ref{eq16}) count elements in $A$ and $B$, respectively.

If an $(r,s)$-near-factorization of $G$ exists, then the system of equations (\ref{eq11})--(\ref{eq16}) has a solution in nonnegative integers.

We again assume that $A$ and $B$ are symmetric. There are four involutions in $G$, namely, 
$(0,0,0)$, $(0,0,1)$, $(0,1,0)$ and $(0,1,1)$, and each $G_{i,j}$ contains one involution. 
It is easy to see that every involution can occur in at most one of the eight sets $A_{i,j}$ ($i,j \in \{0,1\}$) and $B_{i,j}$ ($i,j \in \{0,1\}$).

Computing (\ref{eq11}) $-$ (\ref{eq12}) $+$ (\ref{eq13}) $-$ (\ref{eq14}), we have
\begin{align*}
(a_{0,0}- a_{0,1} + a_{1,0}- a_{1,1})(b_{0,0}-b_{0,1} + b_{1,0}-b_{1,1} )  &= -1 
\end{align*}
Hence,
\begin{equation}
a_{0,0}- a_{0,1} + a_{1,0}- a_{1,1} = \pm 1 \quad\quad  \text{and}  \quad\quad b_{0,0}-b_{0,1} + b_{1,0}-b_{1,1} = \mp 1.
\label{eq20}\end{equation}
Adding the two equations in  (\ref{eq20}) to (\ref{eq15}) and (\ref{eq16}), resp., we see that
that 
\begin{equation} a_{0,0} + a_{1,0} = \frac{r\pm 1}{2} \quad\quad  \text{and}  \quad\quad b_{0,0}  + b_{1,0} = \frac{s \mp 1}{2}.
\label{eq22}
\end{equation}

Now compute
(\ref{eq11}) $+$ (\ref{eq12}) $-$ (\ref{eq13}) $-$ (\ref{eq14}), to get
\begin{align*}
(a_{0,0}+ a_{0,1} - a_{1,0}- a_{1,1})(b_{0,0}+b_{0,1} - b_{1,0}-b_{1,1} )  &= -1 
\end{align*}
Hence,
\begin{equation}
a_{0,0}+ a_{0,1} - a_{1,0}- a_{1,1} = \pm 1 \quad\quad  \text{and}  \quad\quad b_{0,0}+b_{0,1} - b_{1,0}-b_{1,1} = \mp 1.
\label{eq25}\end{equation}
Adding the two equations in  (\ref{eq25}) to (\ref{eq15}) and (\ref{eq16}), resp., we see that
\begin{equation} a_{0,0} +a_{0,1} = \frac{r \pm 1}{2} \quad\quad  \text{and}  \quad\quad b_{0,0}  + b_{0,1} = \frac{s \mp 1}{2}. 
\label{eq27}
\end{equation}

Lastly, compute
(\ref{eq11}) $-$ (\ref{eq12}) $-$ (\ref{eq13}) $+$ (\ref{eq14}), to get
\begin{align*}
(a_{0,0}- a_{0,1} - a_{1,0}+ a_{1,1})(b_{0,0}-b_{0,1} - b_{1,0}+b_{1,1} )  &= -1 
\end{align*}
Hence,
\begin{equation}
a_{0,0}- a_{0,1} - a_{1,0}+ a_{1,1} = \pm 1 \quad\quad  \text{and}  \quad\quad b_{0,0}-b_{0,1} - b_{1,0}+b_{1,1} = \mp 1.
\label{eq30}\end{equation}
Adding the two equations in  (\ref{eq30}) to (\ref{eq15}) and (\ref{eq16}), resp., we see that 
\begin{equation} a_{0,0} + a_{1,1} = \frac{r \pm 1}{2} \quad\quad  \text{and}  \quad\quad b_{0,0}  + b_{1,1} = \frac{s \mp 1}{2}.
\label{eq31}
\end{equation}
Now, from (\ref{eq22}), (\ref{eq27}) and (\ref{eq31}), adding the equations involving the $a_{i,j}$'s, we see that
\[ 3a_{0,0} + a_{0,1} +  a_{1,0} + a_{1,1} = 3\left( \frac{r-1}{2}\right) + \theta, \]
where $\theta \in \{0,1,2,3\}$.
Subtracting (\ref{eq15}), we have
\[ a_{0,0}  = \frac{r-3 + 2\theta}{4}, \]
$\theta \in \{0,1,2,3\}$.
Recalling that $r \equiv 1 \bmod4$, we see that $\theta$ must be odd, so $\theta = 1$ or $3$.
Hence $a_{0,0} = (r-1)/4$ or $(r+3)/4$. 

First, suppose that $a_{0,0} = (r+3)/4$, so $\theta = 3$.
This means that
\[  a_{0,0} + a_{0,1} = a_{0,0} + a_{1,0} = a_{0,0} + a_{1,1} = \frac{r + 1}{2},\]
so 
\[  a_{0,1} =  a_{1,0} = a_{1,1} = \frac{r - 1}{4}.\]
In the other case, $a_{0,0} = (r-1)/4$ and $\theta = 1$.
It turns out that one of $a_{0,1},a_{1,0}$ and $a_{1,1}$ equals $(r+3)/4$ and the other two equal $(r-1)/4$. 
In both cases, three of $a_{0,0}, a_{0,1}, a_{1,0}$ and $a_{1,1}$ equal $(r-1)/4$ and the other equals $(r+3)/4$.

Similar computations show that three of $b_{0,0}, b_{0,1}, b_{1,0}$ and $b_{1,1}$ equal $(s+1)/4$ and the other equals $(s-3)/4$.

Now, if we examine equations (\ref{eq11})--(\ref{eq14}), it is possible to prove that they are satisfied if and only if $a_{i,j} = (r+3)/4$ and  $b_{i,j} = (s-3)/4$ for the same indices $i$ and $j$.

Write $r = 4u+1$ and $s = 4v-1$. Then the values $a_{i,j}$ include
$u$ three times and $u+1$ once, and the values $b_{i,j}$ include
$v$ three times and $v-1$ once. Because  $rs = 4t-1$, we have
$(4u+1)(4v-1) = 4t-1$, so $4uv + v-u = t$. Because $t$ is odd, $u$ and $v$ have different parity.
It follows that, for each ordered pair $(i,j)$, exactly one of $a_{i,j}$ and $b_{i,j}$ is even. The odd values of 
the $a_{i,j}$'s and $b_{i,j}$'s correspond to $A_{i,j}$'s and $B_{i,j}$'s that contain involutions.

\begin{remark}
{\rm The analysis given above applies to any abelian group of the form $H \times (\zed_2)^2$ where $|H|$ is odd. }
\end{remark}

We provide an example to show how the above discussion can greatly reduce the search time for a near-factorization (for certain parameters).

\begin{example}
\label{exam137}
{\rm Suppose we seek a $(13,7)$-near-factorization in $\zed_{23} \times (\zed_2)^2$. The existence of this near-factorization  was ruled out in \cite{Pech} using a computer search, but no details were provided. We have $13 = 4 \times 3 + 1$ and $7 = 4\times 2 - 1$. 
The $a_{i,j}$-values are $3, 3, 3$ and $4$ (they sum to $13$)
and the $b_{i,j}$-values are $2, 2, 2$ and $1$ (they sum to $7$).
Suppose that $a_{i,j} = 4$ and $b_{i,j} = 1$. There are four cases to consider, which correspond to the four possible choices of $(i,j)$. 

Suppose for the purposes of illustration that $(i,j) = (0,0)$ (this is one of the four cases).  
Then $a_{0,0} = 4$, $a_{0,1} = a_{1,0} = a_{1,1} = 3$, $b_{0,0} = 1$ and  $b_{0,1} = b_{1,0} = b_{1,1} = 2$.
Then $(0,0,0) \in B_{0,0}$, $(0,0,1) \in A_{0,1}$, $(0,1,0) \in A_{1,0}$ and $(0,1,1) \in A_{1,1}$. 
The three sets $A_{0,1}$,
$A_{1,0}$ and $A_{1,1}$ each contain one symmetric pair; $A_{0,0}$ contains two symmetric pairs; $B_{0,1}$,
$B_{1,0}$ and $B_{1,1}$ each contain one symmetric pair; and $B_{0,0}$ contains no symmetric pairs.
Note that there are only $11$ symmetric pairs in each $G_{i,j}$, namely,
\[ \bigl\{(1,i,j),(22,i,j)\bigr\}, \bigl\{(2,i,j),(21,i,j)\bigr\}, \dots , \bigl\{(11,i,j),(12,i,j)\bigr\} .\]
So an exhaustive search will not take long at all.

Computationally, the most efficient way would be to start with $B$ (because $|B| < |A|$). To construct $B$, we just need to choose three symmetric pairs---one from $B_{0,1}$, one from 
$B_{1,0}$ and one from $B_{1,1}$. The symmetric pair in $B_{0,1}$ can be assumed to be $\{(1,0,1),(22,0,1)\}$, by applying a (multiplicative) automorphism of $\zed_{23}$. So there are only $11^2 = 121$ possible $B$'s to be considered. For each $B$, we check to see if it has a mate $A$ using Algorithm \ref{alg1} or \ref{alg2}.
 
The other three cases (namely, $(i,j) = (0,1)$, $(1,0)$ and $(1,1)$) would be handled in a similar fashion. They would each require checking only $121$ possible sets $B$. However, we do not need to consider all three of these cases, since they are equivalent under the action of $\mathrm{GL}(2,2)$, as discussed in Section \ref{involutions.sec}. So it suffices to consider only the  two cases  $(i,j) = (0,0)$ and $(i,j) = (0,1)$.
$\hfill\blacksquare$
}
\end{example}

\begin{remark}\label{2.2.remark}
{\rm We have analyzed the structure of possible near-factorizations in the two groups $\zed_{47} \times (\zed_2)^2$ and $\zed_{49} \times (\zed_2)^2$ in a similar fashion. Indeed, these analyses were used to greatly speed up exhaustive searches that are reported in Table \ref{nonexistence.tab}.}
\end{remark}

\section{Nonexistence results for near-factorizations in noncyclic abelian groups}
\label{nonexist.sec}

We are interested in nonexistence results for near-factorizations in noncyclic abelian groups. It is known (mainly from results in de Caen {\it et al.} \cite{CGHK}) that there is no near-factorization in a noncyclic abelian group of order at most $100$. We review the proof of this result and explore some larger possibilities now.

The following  nonexistence results from \cite{CGHK} are very useful.

\begin{theorem}[D. de Caen {\it et al.}, {\cite[Proposition 3]{CGHK}}]
\label{CGHK-prop}
Suppose $G$ is an abelian group and $G \setminus \{e\} = AB$ is a near-factorization with $|A| \leq 4$. Then $G$ must be a cyclic group.
\end{theorem}

\begin{corollary}
\label{3p+1}
Suppose $n = 3p+1$ where $p$ is prime. Then there does not exist a noncyclic abelian group of order $n$ having a nontrivial near-factorization.
\end{corollary}

Note that, if $p \equiv 1 \bmod 4$ in Corollary \ref{3p+1}, then $n \equiv 0 \bmod 4$ and there exists at least one noncyclic abelian group of order $n$. For $100 \leq n \leq 400$, Corollary \ref{3p+1} rules out the values
\[n = 112, 124, 160, 184, 220, 250, 268, 292, 304, 328 \text{ and } 340.\] All of these values of $n$ are divisible by 4, except for $n = 250 = 2 \times 5^3$. However, there are noncyclic abelian groups of order $250$ because $250$ is divisible by $5^2$.

\begin{theorem}[D. de Caen {\it et al.}, {\cite[Corollary 2]{CGHK}}]
\label{CGHK-cor}
Suppose $G$ is a \textup{(}not necessarily abelian\textup{)} group that has a near-factorization $G \setminus \{e\} = AB$. Suppose also that there is homomorphism of $G$ onto an abelian group $H$ of exponent $2$, $3$, $4$ or $6$. Then $|A| \geq |H|- 1$.
\end{theorem}

\begin{theorem}[D. de Caen {\it et al.}, {\cite[Example 7]{CGHK}}]
\label{special}
Suppose $G = (\zed_2)^n, (\zed_3)^n, (\zed_4)^n, (\zed_3)^m \times (\zed_2)^n$ or $(\zed_2)^n \times (\zed_4)^m$, where $n$ and $m$ are positive integers. Then $G$ does not have a nontrivial near-factorization.
\end{theorem}

\begin{theorem}[D. de Caen {\it et al.}, {\cite[Theorem 3]{CGHK}}]
\label{quotient}
If the group $G$ has the elementary abelian group $(\zed_p)^m$ as a quotient group, then for any near-factorization 
$G \setminus \{e\} = AB$, it holds that
\[
\begin{array}{ll}
|A|^{p-1} \equiv |B|^{p-1} \equiv 1 \bmod p^m & \text{if $p$ is an odd prime}\\
|A| \equiv -|B| \equiv \pm1 \bmod 2^m & \text{if $p=2$.}
\end{array}
\]
\end{theorem}

Here is another nonexistence result we will use, due to P\^{e}cher \cite{Pech}.

\begin{theorem}[A. P\^{e}cher, {\cite[Lemma 7]{Pech}}]
\label{Pech.thm}
Suppose $G = \zed_{2m} \times \zed_{4n} \times G'$, where $m$ and $n$ are positive integers and $G'$ is an abelian group. Then $G$ does not have  a near-factorization $(A,B)$ if $|A| \equiv \pm 3 \bmod 8$ or $|B| \equiv \pm 3 \bmod 8$. 
\end{theorem}

All the noncyclic abelian groups of order at most $200$ are listed in Table \ref{nonexistence.tab}, along with nonexistence results for all the possible nontrivial  near-factorizations. We consider all possible integer factorizations $n - 1 = rs$ with $1 < r \leq s$.  We note that the exhaustive searches for $(11,17)$- and $(13,15)$-near-factorizations are accelerated using the analysis mentioned in Remark \ref{2.2.remark}.

\newpage

\begin{center}
\renewcommand{\arraystretch}{1.2}
\begin{longtable}{c|c|c|c|c|c|p{2in}}
\caption{\mbox{Nonexistence of near-factorizations in noncyclic abelian groups}}\\ \hline
\label{nonexistence.tab}
$n$ & $n =$ & $n-1=$ & $G$ & $r$ & $s$  & {authority}
\\ \hline 
\endfirsthead
\caption[]{(continued)}\\ \hline
$n$ & $n=$  & $n-1=$ & $G$ & $r$ & $s$  & {authority}
\\ \hline 
\endhead
\hline
\endfoot
$9$ & $3^2$  & $2^1  4^1$ & $(\zed_3)^2$ & $2$ & 4   & Theorem\   \ref{CGHK-prop}\\
$16$ & $2^4$  & $3^1  5^1$ & all & $3$ & $5$   & Theorem\   \ref{CGHK-prop}\\
$25$ & $5^2$  & $2^3  3^1$ & $(\zed_5)^2$ & \text{all} & \text{all}   & \text{Theorem\   \ref{CGHK-prop}}\\
$27$ & $3^3$  & $2^1  13^1$ & \text{all} & 2 & 13   & \text{Theorem\   \ref{CGHK-prop}}\\
$28$ & $2^2  7^1$ & $3^3$ & $(\zed_7)^2 \times \zed_2$ & 3 & 9 & \text{Theorem\   \ref{CGHK-prop}}\\
$36$ & $2^2 3^2$ & $5^1  7^1$ & $(\zed_3)^2 \times   (\zed_2)^2$ & 5 & 7  & Theorem\   \ref{CGHK-cor}, $H= (\zed_3)^2 \times   (\zed_2)^2$ \\
$36$ & $2^2 3^2$ & $5^1  7^1$ & $\zed_9 \times (\zed_2)^2$ & 5 & 7  & Theorem\   \ref{CGHK-cor}, $H= \zed_3 \times   (\zed_2)^2$ \\
$36$ & $2^2 3^2$ & $5^1  7^1$ & $(\zed_3)^2 \times  \zed_4$ & 5 & 7  & Theorem\   \ref{CGHK-cor}, $H= (\zed_3)^2 \times   (\zed_2)$ \\
$40$ & $2^3 5^1$ & $3^1  13^1$ & \text{all} & 3 & 13 & {Theorem\   \ref{CGHK-prop}}\\
$45$ & $3^2 5^1$  & $2^2  11^1$ & $\zed_{5} \times (\zed_3)^2$ & \text{all} & \text{all}   & {Theorem\   \ref{CGHK-prop}}\\
$49$ & $7^2$  & $2^4  3^1$ & $(\zed_7)^2$ & \text{all} & \text{all}   & \cite[Example 5]{CGHK}\\
$50$ & $2^1 5^2$ & $7^2$ & $(\zed_5)^2\times \zed_2$ & 7 & 7 & \cite[Example 6]{CGHK}\\
$52$ & $2^2  13^1$ & $3^1  17^1$ & $\zed_{13} \times (\zed_2)^2$ & 3 & 17 & {Theorem\   \ref{CGHK-prop}}\\
$56$ & $2^3  7^1$ & $5^1  11^1$ & $\zed_7 \times \zed_2 \times \zed_4$ & 5 & 11 & Theorem\   \ref{CGHK-cor}, $H= \zed_2 \times \zed_4$\\
$56$ & $2^3 7^1$ & $5^1  11^1$ & $\zed_7 \times (\zed_2)^3$ & 5 & 11 & Theorem\   \ref{CGHK-cor}, $H= (\zed_2)^3$\\
$63$ & $3^2 7^1$ & $2^1 31^1$ & $\zed_7 \times (\zed_3)^2$ & 2 & 31  & {Theorem\   \ref{CGHK-prop}}\\
$64$ & $2^6$ & $3^2  7^1$ & $(\zed_2)^2 \times \zed_{16}$ & 7 & 9  & Theorem\   \ref{CGHK-cor}, $H= (\zed_2)^2 \times \zed_4$\\
$64$ & $2^6$ & $3^2  7^1$ & $(\zed_8)^2$  & 7 & 9  & Theorem\   \ref{CGHK-cor}, $H= (\zed_4)^2$\\
$64$ & $2^6$ & $3^2  7^1$ & $\zed_4 \times \zed_{16}$  & 7 & 9  & Theorem\   \ref{CGHK-cor}, $H= (\zed_4)^2$\\
$64$ & $2^6$ & $3^2  7^1$ & $(\zed_2)^6$ & 7 & 9  & Theorem\   \ref{special}\\
$64$ & $2^6$ & $3^2  7^1$ & $(\zed_4)^3$ & 7 & 9  & Theorem\   \ref{special}\\
$64$ & $2^6$ & $3^2  7^1$ & $(\zed_2)^4 \times \zed_4$ & 7 & 9  & Theorem\   \ref{special}\\
$64$ & $2^6$ & $3^2  7^1$ & $\zed_{2} \times \zed_{32}$ & 7 & 9  & \cite[Example 9 and Prop.\ 5]{CGHK}\\
$64$ & $2^6$ & $3^2  7^1$ & $\zed_2 \times \zed_4 \times \zed_8$ & 7 & 9  & Theorem\   \ref{CGHK-cor}, $H= \zed_2 \times (\zed_4)^2$\\
$75$ & $3^1 5^2$ & $2^1 37^1$ & $(\zed_5)^2 \times \zed_3$ & 2 & 37  & {Theorem\   \ref{CGHK-prop}}\\
$76$ & $2^2   19^1$ & $3^1 5^2$ & $\zed_{19} \times (\zed_2)^2$ & 5 & 15  & \cite[p.\ 62]{CGHK} and  \cite{Pech} \\
$81$ & $3^4$ & $2^4 5^1$ & \text{all} & 2 & 40  & Theorem\   \ref{CGHK-prop}\\
$81$ & $3^4$ & $2^4 5^1$ & \text{all} & 4 & 20  & Theorem\   \ref{CGHK-prop}\\
$81$ & $3^4$ & $2^4 5^1$ & $(\zed_3)^4$ & all & all  & Theorem\   \ref{special}\\
$81$ & $3^4$ & $2^4 5^1$ & $(\zed_3)^2 \times \zed_9$ & all & all  & Theorem\   \ref{CGHK-cor}\\
$81$ & $3^4$ & $2^4 5^1$ & $\zed_3 \times \zed_{27}$ & all & all  &  \cite[Prop.\ 5]{CGHK}\\
$81$ & $3^4$ & $2^4 5^1$ & $(\zed_9)^2 $ & all & all  &  \cite[Prop.\ 5]{CGHK}\\
$88$ & $2^3   11^1$ & $3^1 29^1$ & \text{all} & 3 & 29  & Theorem\   \ref{CGHK-prop}\\
$92$ & $2^2 23^1$ & $7^1 13^1$ & $\zed_{23} \times (\zed_2)^2$ & 7 & 13  & \cite{Pech} \\
$96$ & $2^5 3^1$ & $5^1 19^1$ & $\zed_3 \times (\zed_2)^5$ & 5 & 19  & Theorem\   \ref{special}\\
$96$ & $2^5 3^1$ & $5^1 19^1$ & $\zed_3 \times (\zed_2)^3 \times \zed_4$ & 5 & 19  & Theorem\ \ref{CGHK-cor}, 
$H= \zed_3 \times (\zed_2)^4$\\
$96$ & $2^5 3^1$ & $5^1 19^1$ & $\zed_3 \times (\zed_2)^2 \times \zed_8$ & 5 & 19  & Theorem\ \ref{CGHK-cor}, 
$H= \zed_3 \times (\zed_2)^3$\\
$96$ & $2^5 3^1$ & $5^1 19^1$ & $\zed_3 \times \zed_4 \times \zed_8$ & 5 & 19  & Theorem\ \ref{CGHK-cor}, 
$H= \zed_3 \times (\zed_2)^3$\\
$96$ & $2^5 3^1$ & $5^1 19^1$ & $\zed_3 \times \zed_2 \times \zed_{16}$ & 5 & 19  & Theorem\ \ref{CGHK-cor}, 
$H= \zed_3 \times (\zed_2)^3$\\
$96$ & $2^5 3^1$ & $5^1 19^1$ & $\zed_3 \times \zed_2 \times (\zed_4)^2$ & 5 & 19  & Theorem\ \ref{CGHK-cor}, $H= (\zed_2)^3 \times \zed_3$\\
$99$ & $3^2 11^1$ & $2^1 7^2$ & $\zed_{11} \times (\zed_3)^2$ & \text{all} & \text{all}  & Theorem\   \ref{CGHK-cor}, $H= (\zed_3)^2$\\
$100$ & $2^2 5^2$ & $3^2 11^1$ & \text{all} & 3 & 33  & Theorem\   \ref{CGHK-prop}\\
$100$ & $2^2 5^2$ & $3^2 11^1$ & $\zed_{25} \times (\zed_2)^2$ & 9 & 11  & \cite{Pech}\\
$100$ & $2^2 5^2$ & $3^2 11^1$ & $(\zed_{5})^2 \times (\zed_2)^2$ & 9 & 11  & Theorem\   \ref{quotient}, $p = 5$, $m = 2$ \\
$112$ & $2^4  7^1$ & $3^1 37^1$ & \text{all} & 3 & 37  & Corollary \ref{3p+1}\\
$116$ & $2^2  29^1$ & $5^1 23^1$ & $\zed_{29} \times (\zed_2)^2$ & 5 & 23  & exhaustive computer search\\
$117$ & $3^2  13^1$ & $2^2 29^1$ & $\zed_{13} \times (\zed_3)^2$ & \text{all} & \text{all}  & \text{Theorem\   \ref{CGHK-prop}}\\
$120$ & $2^2  3^1 5^1$ & $7^1 17^1$ & $\zed_{15}\times \zed_2 \times \zed_4$ & 7 & 17  & Theorem\   \ref{CGHK-cor}, 
$H= \zed_3 \times \zed_2 \times \zed_2 $\\
$120$ & $2^2  3^1 5^1$ & $7^1 17^1$ & $\zed_{15} \times (\zed_2)^3$ & 7 & 17  & Theorem\   \ref{CGHK-cor}, 
$H= \zed_3 \times (\zed_2)^3$\\
$121$ & $11^2$  & $2^3  3^1 5^1$ & $(\zed_{11})^2$ & \text{all} & \text{all}   & \cite[Example 5]{CGHK}\\
$124$ & $2^2  31^1$ & $3^1 41^1$ & $\zed_{31} \times (\zed_2)^2$ & 3 & 41  & Corollary \ref{3p+1}\\
$125$ & $5^3$ & $2^2 31^1$ & \text{all} & \text{all} & \text{all}  & \text{Theorem\   \ref{CGHK-prop}}\\
$126$ & $2^1 3^2 7^1$ & $5^3$ & $\zed_7 \times (\zed_3)^2 \times  \zed_{2}$ & 5 & 24  & 
Theorem\   \ref{CGHK-cor}, 
$H= (\zed_3)^2 \times \zed_2$\\
$135$ & $3^3 5^1$ & $2^1 67^1$ & \text{all} & 2 & 67  & \text{Theorem\   \ref{CGHK-prop}}\\
$136$ & $2^3 17^1$ & $3^3 5^1$ & all & $3$ & $45$  & \text{Theorem\   \ref{CGHK-prop}}\\
$136$ & $2^3 17^1$ & $3^3 5^1$ & $\zed_{17} \times (\zed_2)^3$ & 5 & 27  & Theorem\   \ref{CGHK-cor}, $H= (\zed_2)^3$ \\
$136$ & $2^3 17^1$ & $3^3 5^1$ & $\zed_{17} \times (\zed_2)^3$ & 9 & 15  & exhaustive computer search\\
$136$ & $2^3 17^1$ & $3^3 5^1$ & $\zed_{17} \times \zed_2 \times  \zed_4$ & 5 & 27  & exhaustive computer search\\
$136$ & $2^3 17^1$ & $3^3 5^1$ & $\zed_{17} \times \zed_2 \times  \zed_4$ & 9 & 15  &  exhaustive computer search\\
$144$ & $2^4 3^2$ & $11^1 13^1$ & all & 11 & 13  & see Lemma \ref{144.lem}\\
$147$ & $3^1 7^2$ & $2^1 73^1$ &  $(\zed_{7})^2 \times \zed_3$ & 2 & 73  & \text{Theorem\   \ref{CGHK-prop}}\\
$148$ & $2^2 37^1$ & $3^1 7^2$ &  $\zed_{37} \times (\zed_2)^2$ & 3 & 49  & \text{Theorem\   \ref{CGHK-prop}}\\
$148$ & $2^2 37^1$ & $3^1 7^2$ &  $\zed_{37} \times (\zed_2)^2$ & 7 & 21  & exhaustive computer search \\
$153$ & $3^2 17^1$ & $2^3 19^1$ & $\zed_{17} \times (\zed_3)^2$ & 2 & 76  & \text{Theorem\   \ref{CGHK-prop}}\\
$153$ & $3^2 17^1$ & $2^3 19^1$ & $\zed_{17} \times (\zed_3)^2$ & 4 & 38  & \text{Theorem\   \ref{CGHK-prop}}\\
$153$ & $3^2 17^1$ & $2^3 19^1$ & $\zed_{17} \times (\zed_3)^2$ & 8 & 19  & exhaustive computer search\\
$156$ & $3^1 2^2 13^1$ & $5^1 31^1$ & $\zed_{39} \times (\zed_2)^2$ & 5 & 31  & exhaustive computer search\\
$160$ & $2^5 5^1$ & $3^1 53^1$ & \text{all} & 3 & 53  & Corollary \ref{3p+1} \\
$162$ & $2^1 3^4$ & $7^1 23^1$ &  $(\zed_{3})^4 \times \zed_2$ & 7 & 23  & Theorem\   \ref{CGHK-cor}, 
$H= (\zed_{3})^4$\\
$162$ & $2^1 3^4$ & $7^1 23^1$ &  $(\zed_{3})^2 \times  \zed_{9} \times \zed_2$ & 7 & 23  & Theorem\   \ref{CGHK-cor}, 
$H= (\zed_{3})^3$\\
$162$ & $2^1 3^4$ & $7^1 23^1$ &  $\zed_{3} \times  \zed_{27} \times \zed_2$ & 7 & 23  & Theorem\   \ref{CGHK-cor}, 
$H= (\zed_{3})^2$\\
$162$ & $2^1 3^4$ & $7^1 23^1$ &  $(\zed_{9})^2 \times \zed_2$ & 7 & 23  & Theorem\   \ref{CGHK-cor}, 
$H= (\zed_{3})^2$\\
$169$ & $13^2$ & $2^3 3^1 7^1$ & $(\zed_{13})^2$ & all & all  & \cite[Example 5]{CGHK}\\
$171$ & $3^2 19^1$ & $2^1 5^1 17^1$ & $\zed_{19} \times (\zed_3)^2$ & 2 & 85  & \text{Theorem\   \ref{CGHK-prop}} \\
$171$ & $3^2 19^1$ & $2^1 5^1 17^1$ & $\zed_{19} \times (\zed_3)^2$ & 5 & 34  & Theorem\   \ref{CGHK-cor}, $H= (\zed_{3})^2$\\
$171$ & $3^2 19^1$ & $2^1 5^1 17^1$ & $\zed_{19} \times (\zed_3)^2$ & 10 & 17  & exhaustive computer search\\
$172$ & $2^2 43^1$ & $3^2 19^1$ & $\zed_{43} \times (\zed_2)^2$ & 3  & 57  & \text{Theorem\   \ref{CGHK-prop}}\\
$172$ & $2^2 43^1$ & $3^2 19^1$ & $\zed_{43} \times (\zed_2)^2$ & 9 & 19  & exhaustive computer search\\
$175$ & $5^2 7^1$ & $2^1 3^1 29 ^1$ & $\zed_{7} \times (\zed_5)^2$ & 2 & 87  & \text{Theorem\   \ref{CGHK-prop}}\\
$175$ & $5^2 7^1$ & $2^1 3^1 29 ^1$ & $\zed_{7} \times (\zed_5)^2$ & 3 & 58  & \text{Theorem\   \ref{CGHK-prop}}\\
$175$ & $5^2 7^1$ & $2^1 3^1 29 ^1$ & $\zed_{7} \times (\zed_5)^2$ & 6 & 29  & exhaustive computer search \\
$176$ & $2^4 11^1$ & $5^2 7^1$ & $\zed_{11} \times (\zed_2)^4 $ & all & all  & Theorem\   \ref{CGHK-cor}, $H= (\zed_{2})^4$\\
$176$ & $2^4 11^1$ & $5^2 7^1$ & $\zed_{11} \times (\zed_2)^2 \times \zed_4$ & all & all  & Theorem\   \ref{CGHK-cor}, 
$H= (\zed_2)^2 \times \zed_4$\\
$176$ & $2^4 11^1$ & $5^2 7^1$ & $\zed_{11} \times \zed_2 \times \zed_8$ & 5 & 35  & Theorem\   \ref{CGHK-cor}, 
$H= \zed_2 \times \zed_4$\\
$176$ & $2^4 11^1$ & $5^2 7^1$ & $\zed_{11} \times \zed_2 \times \zed_8$ & 7 & 25  & exhaustive computer search\\
$176$ & $2^4 11^1$ & $5^2 7^1$ & $\zed_{11} \times (\zed_4)^2$ &  all & all  & Theorem\   \ref{CGHK-cor}, $H= (\zed_{4})^2$\\
$184$ & $2^3 23^1$ & $3^1 61^1$ & \text{all} & 3 & 61  & Corollary \ref{3p+1} \\
$188$ & $2^2 47^1$ & $11^1 17^1$ & $\zed_{47} \times (\zed_2)^2$ & 11 & 17   & exhaustive computer search\\
$189$ & $3^3 7^1$ & $2^2 47^1$ & all & all & all  & \text{Theorem\   \ref{CGHK-prop}} \\
$196$ & $2^2 7^2$ & $3^1 5^1 13^1 $ & all & 3 & 65  & \text{Theorem\   \ref{CGHK-prop}}\\

$196$ & $2^2 7^2$ & $3^1 5^1 13^1 $ & $(\zed_7)^2 \times (\zed_2)^2$ & 5 & 39  & Theorem\   \ref{quotient}, $p = 7$, $m = 2$\\
$196$ & $2^2 7^2$ & $3^1 5^1 13^1 $ & $(\zed_7)^2 \times (\zed_2)^2$ & 13 & 15  & Theorem\   \ref{quotient}, $p = 7$, $m = 2$\\
$196$ & $2^2 7^2$ & $3^1 5^1 13^1 $ & $(\zed_7)^2 \times \zed_4$ & 5 & 39  & Theorem\   \ref{quotient}, $p = 7$, $m = 2$\\
$196$ & $2^2 7^2$ & $3^1 5^1 13^1 $ & $(\zed_7)^2 \times \zed_4$ & 13 & 15  & Theorem\   \ref{quotient}, $p = 7$, $m = 2$\\
$196$ & $2^2 7^2$ & $3^1 5^1 13^1 $ & $\zed_{49} \times (\zed_2)^2$ & 5 & 39  & exhaustive computer search\\
$196$ & $2^2 7^2$ & $3^1 5^1 13^1 $ & $\zed_{49} \times (\zed_2)^2$ & 13 & 15  & exhaustive computer search\\
\end{longtable}
\renewcommand{\arraystretch}{1}
\end{center}



\begin{lemma}
\label{144.lem}
There does not exist a nontrivial near-factorization in a noncyclic abelian group of order $144$.
\end{lemma}
\begin{proof}
We have $n = 144 = 9 \times 16$ and $n-1 = 11 \times 13$, so we are seeking an $(11,13)$-near-factorization. The $3$-Sylow subgroup of an abelian group of order $144$ is $\zed_9$ or $(\zed_3)^2$; and the $2$-Sylow subgroup is 
$\zed_{16}$, 
$\zed_2 \times \zed_8$, 
$(\zed_4)^2$, 
$(\zed_2)^2 \times \zed_4$, or 
$(\zed_2)^4$.
There are therefore ten  abelian groups of order $144$, of which one is cyclic.
Seven of the nine noncyclic possibilities are ruled out using Theorem \ref{CGHK-cor}, as indicated in Table \ref{144.tab} (note that the relevant groups $H$ all have exponent equal to six).
The remaining two possibilities are $G = \zed_9 \times \zed_2 \times \zed_8$ and $G = \zed_9 \times \zed_4 \times \zed_4$.
Both of these cases are impossible, as a consequence of Theorem \ref{Pech.thm}.
\end{proof}

\renewcommand{\arraystretch}{1.2}
\begin{table}[htb]
\caption{Nonexistence of near-factorizations in noncyclic abelian groups of order $144$}
\label{144.tab}
\[
\begin{array}{c|c|c|c}
\text{$3$-Sylow subgroup} & \text{$2$-Sylow subgroup} & H  & |H|  \\ \hline
\zed_9 & (\zed_2)^2 \times \zed_4  & \zed_3 \times  (\zed_2)^3   & 24 \\
\zed_9 & (\zed_2)^4 & \zed_3 \times   (\zed_2)^4  & 48\\
(\zed_3)^2 & \zed_{16} & (\zed_3)^2 \times  \zed_{2}  & 18\\
(\zed_3)^2 & \zed_2 \times \zed_8 & (\zed_3)^2 \times  (\zed_{2})^2  & 36\\
(\zed_3)^2 & (\zed_4)^2 & (\zed_3)^2 \times  (\zed_{2})^2  & 36\\
(\zed_3)^2 & (\zed_2)^2 \times \zed_4 & (\zed_3)^2 \times  (\zed_{2})^3  & 72\\
(\zed_3)^2 & (\zed_2)^4 & (\zed_3)^2 \times   (\zed_{2})^4  & 144
\end{array}
\]
\end{table}
\renewcommand{\arraystretch}{1}

\section{Nonexistence of certain strong CEDFs}
\label{SCEDF.sec}

In this section, we  provide an alternate proof of a nonexistence theorem due to Wu, Yang and Feng \cite{WYF} that involves strong circular external difference families. We first need to define these objects.

For two disjoint subsets $A,B$ of a multiplicative group $(G,\cdot)$, define the multiset $\mathcal{D}(B,A)$ as follows:
\[ \mathcal{D}(B,A) = BA^{-1} = \{ yx^{-1}\colon  y \in B, x \in A \}.  \]

\begin{definition}[Strong circular external difference family (SCEDF)]
Let $G$ be a \textup{(}multiplicative\textup{)} group of order $n$ having identity $e$. An \emph{$(n, m, \ell; \lambda)$-strong circular external difference family} in $G$ \textup{(}or $(n, m, \ell; \lambda)$-SCEDF\textup{)} is a set of $m$ disjoint $\ell$-subsets of $G$, say $\mathcal{A} = \{A_0,\dots,A_{m-1}\}$, such that the following multiset equation holds for every $j$, $0 \leq j \leq m-1$:
\[
\mathcal{D}(A_{j+1\bmod m}, A_j) = \lambda (G \setminus \{e\}).
\]
\end{definition}
We observe that $\ell^2 = \lambda (n-1)$ if an $(n, m, \ell; \lambda)$-SCEDF exists.
SCEDFs were introduced in \cite{VS}. SCEDFs and related objects have been studied further in \cite{PS,WYF}.

\begin{theorem}[J. Wu {\it et al.}, {\cite[Theorem 4]{WYF}}]
\label{nonexist.thm}
There does not exist an $(n, m, \ell; 1)$-SCEDF with $m > 2$.
\end{theorem}

\begin{proof}
Suppose $A_1,A_2$ and $A_3$ are three consecutive sets in an SCEDF in a multiplicative group $G$ having identity $e$.
Then $A_1 (A_2)^{-1} = A_2 (A_3)^{-1} = G \setminus \{e\}$. Clearly $(A_2 (A_3)^{-1})^{-1} = A_3 (A_2)^{-1}$. Hence,
$A_1 (A_2)^{-1} = A_3 (A_2)^{-1} = G \setminus \{e\}$, so $(A_1, (A_2)^{-1})$ and $(A_3, (A_2)^{-1})$ are both near-factorizations of $G$.
It follows from Theorem \ref{T3} (which states that mates are unique) that $A_1 = A_3$. However, this is impossible because the sets in a SCEDF are required to be disjoint.
\end{proof}

The proof of Theorem \ref{nonexist.thm} given in \cite{WYF} uses the group ring $\zed[G]$; our proof is just a consequence of the formula (\ref{eq2}). We also comment that the result proven in \cite{WYF} is only stated for abelian groups. 

\section{Near-factorizations with $\lambda > 1$}
\label{lambda.sec}

One possible generalization of a near-factorization would consist an ordered pair of sets $(A,B)$ in a group $G$ satisfying the multiset equation $AB = \lambda(G \setminus \{e\})$, where $\lambda$ is a positive integer. We  refer to this as a near-factorization with \emph{index} $\lambda$. If $|A| = r$ and $|B| = s$, then we will use the notation
$(r,s,\lambda)$-near-factorization. In a manner analogous to Theorem \ref{equiv.thm}, a near-factorization with index $\lambda$ would give rise to a matrix factorization $XY = \lambda(J-I)$.  If $X$ is invertible, then we can solve for $Y$ as a function of $X$, as we did in the case $\lambda = 1$. Thus mates of $A$ are unique for near-factorizations with index $\lambda$ (the reader can fill in the details). One consequence is \ that
the conclusion of Theorem \ref{nonexist.thm} also holds for $(n, m, \ell; \lambda)$-SCEDFs with $\lambda > 1$. 

We observe that there exist near-factorizations with index $\lambda >1$ for certain noncyclic abelian groups.
Suppose $q = p^j$ where $p$ is prime and $q \equiv 1 \bmod 4$. 
It  follows from results of Huczynska and Paterson \cite[Example 5.1]{HP} that there is a near-factorization of the elementary abelian group $(\zed_p)^j$ having index $\lambda = (q-1)/4$. This near-factorization can be constructed as $(A,B)$, where $A$ consists of the quadratic residues in $\eff_q$ and $B$ consists of the quadratic nonresidues; hence, 
we have a $((q-1)/2,(q-1)/2,(q-1)/4)$-near-factorization. (Actually it is shown in \cite{HP} that these sets $A$ and $B$ yield an $(q, 2, (q-1)/2; (q-1)/4)$-SCEDF. Because $-1$ is a quadratic residue in $\eff_q$ if $q \equiv 1 \bmod 4$, it follows  that $(A,B)$ is also a $((q-1)/4)$-near-factorization.) The smallest example of this construction yields a $(4,4,2)$-near-factorization of $(\zed_3)^2$.

We have done some exhaustive searches for $(r,s,2)$-near-factorizations in some small noncyclic abelian groups and it turns out that examples are not hard to find. See Table \ref{T3.tab} for a listing of some $(r,s,2)$-near-factorizations in noncyclic abelian groups of odd order $n \leq 81$. It is interesting to note that all the examples we have found have $r = 4$, and we have not found any 
$(r,s,2)$-near-factorizations in noncyclic abelian groups of even order.

\begin{center}
\renewcommand{\arraystretch}{1.2}
\begin{longtable}{@{}l|r|r|l|l@{}}
\caption{Index $2$ near-factorizations of abelian noncyclic groups of odd order $n \leq 81$}\\ \hline
\label{T3.tab}
group & $r$ & $s$  & $A$ & $B$  \\ \hline
\endfirsthead
\caption[]{(continued)}\\ \hline
group & $r$ & $s$  & $A$ & $B$  \\ \hline
\endhead
$(\zed_3)^2 $ & $ 4 $ & $ 4  $ & $ \{(0,1),(1,0),(0,2),(2,0)\} $ & $ \{(1,1),(1,2),(2,1),(2,2)\}$\\ \hline
$(\zed_5)^2 $ & $ 4 $ & $ 12  $ & $ 
\{(0,1),(1,0),(0,4),(4,0)\}
	 $ & $ 
	 \begin{array}{@{}l@{}}
\{%
(1,1),(1,2),(1,3),(1,4),(2,1),(2,4),\\
(3,1),(3,4),(4,1),(4,2),(4,3),(4,4)
\}\end{array}$\\ \hline
$\zed_3 \times \zed_9 $ & $ 4 $ & $ 13$ & $
 \{(0,1),(1,1),(0,8),(2,8)\} $ & $
	\begin{array}{@{}l@{}}
	\{
(0,0),(0,2),(0,7),(1,3),(1,5),(1,6),\\
(1,7),(1,8),(2,1),(2,2),(2,3),(2,4),\\
(2,6)\}
	\end{array}$\\ \hline
$\zed_5 \times (\zed_{3})^2  $ & $ 4 $ & $ 22$ & $
\{(1,0,1),(1,1,0),(4,0,2),(4,2,0)\}
$ & $
	\begin{array}{@{}l@{}}
\{%
(0,1,1),(0,1,2),(0,2,1),(0,2,2),\\
(1,0,0),(1,0,2),(1,2,0),(1,2,2),\\
(2,0,1),(2,1,0),(2,1,1),(2,1,2),\\
(2,2,1),(3,0,2),(3,1,2),(3,2,0),\\
(3,2,1),(3,2,2),(4,0,0),(4,0,1),\\
(4,1,0),(4,1,1)%
\}
\end{array}$\\ \hline
$(\zed_7)^2  $ & $ 4 $ & $ 24$ & $
 \{(0,1),(1,0),(0,6),(6,0)\} $ & $
	\begin{array}{@{}l@{}}
\{%
(1,1),(1,2),(1,3),(1,4),(1,5),(1,6),\\
(2,1),(2,6),(3,1),(3,3),(3,4),(3,6),\\
(4,1),(4,3),(4,4),(4,6),(5,1),(5,6),\\
(6,1),(6,2),(6,3),(6,4),(6,5),(6,6)\}
\end{array}$\\ \hline
$\zed_7 \times (\zed_{3})^2  $ & $ 4 $ & $ 31$ & $
\{(1,0,1),(1,1,0),(6,0,2),(6,2,0)\}
$ & $
	\begin{array}{@{}l@{}}
\{%
(0,0,0),(0,0,1),(0,0,2),(0,1,0),\\
(0,2,0),(1,1,1),(1,1,2),(1,2,1),\\
(1,2,2),(2,0,0),(2,0,2),(2,2,0),\\
(2,2,2),(3,0,1),(3,1,0),(3,1,1),\\
(3,1,2),(3,2,1),(4,0,2),(4,1,2),\\
(4,2,0),(4,2,1),(4,2,2),(5,0,0),\\
(5,0,1),(5,1,0),(5,1,1),(6,1,1),\\
(6,1,2),(6,2,1),(6,2,2)\}
\end{array}$\\ \hline
$(\zed_5)^2 \times \zed_{3}  $ & $ 4 $ & $ 37$ & $
\{(0,1,1),(1,0,1),(0,4,2),(4,0,2)\}
$ & $
	\begin{array}{@{}l@{}}
\{%
(0,0,0),(0,1,2),(0,2,2),(0,3,1),\\
(0,4,1),(1,0,2),(1,1,0),(1,1,1),\\
(1,2,0),(1,2,2),(1,3,0),(1,3,2),\\
(1,4,1),(1,4,2),(2,0,2),(2,1,0),\\
(2,1,2),(2,2,1),(2,3,0),(2,4,0),\\
(2,4,1),(3,0,1),(3,1,0),(3,1,2),\\
(3,2,0),(3,3,2),(3,4,0),(3,4,1),\\
(4,0,1),(4,1,1),(4,1,2),(4,2,0),\\
(4,2,1),(4,3,0),(4,3,1),(4,4,0),\\
(4,4,2)\}
\end{array}$\\ \hline
$(\zed_9)^2  $ & $ 4 $ & $ 40$ & $
 \{(0,1),(1,0),(0,8),(8,0)\} $ & $
	\begin{array}{@{}l@{}}
\{%
(1,1),(1,2),(1,3),(1,4),(1,5),(1,6),\\
(1,7),(1,8),(2,1),(2,8),(3,1),(3,3),\\
(3,4),(3,5),(3,6),(3,8),(4,1),(4,3),\\
(4,6),(4,8),(5,1),(5,3),(5,6),(5,8),\\
(6,1),(6,3),(6,4),(6,5),(6,6),(6,8),\\
(7,1),(7,8),(8,1),(8,2),(8,3),(8,4),\\
(8,5),(8,6),(8,7),(8,8)\}
\end{array}$\\ \hline
$\zed_3 \times \zed_{27}  $ & $ 4 $ & $ 40$ & $
 \{(0,1),(1,1),(0,26),(2,26)\} $ & $
	\begin{array}{@{}l@{}}
\{%
(0,4),(0,6),(0,7),(0,8),(0,9),(0,11),\\
(0,16),(0,18),(0,19),(0,20),(0,21),\\
(0,23),(1,0),(1,2),(1,3),(1,4),(1,5),\\
(1,7),(1,12),(1,14),(1,15),(1,16),\\
(1,17),(1,19), (1,24),(1,26),(2,0),\\
(2,1),(2,3),(2,8),(2,10),(2,11),\\
(2,12),(2,13),(2,15),(2,20),(2,22),\\
(2,23),(2,24),(2,25)\}
\end{array}$\\ \hline
\end{longtable}
\renewcommand{\arraystretch}{1}
\end{center}

\section{Summary and Discusion}
\label{summary.sec}

We have shown that, if a set $A$ has a mate (in a near-factorization), then the mate is unique. Moreover, there is a simple direct formula to compute this mate.  This yields more efficient exhaustive searches for near-factorizations, and enables larger groups to be searched. We have examined all  noncyclic abelian groups of order less than 200. In every case, existence of a nontrivial near-factorization is ruled out by one of several theoretical criteria or by an exhaustive computer search. Hence there is strong empirical evidence that nontrivial near-factorizations do not exist in noncyclic abelian groups. On the other hand, nontrivial near factorizations of index $\lambda > 1$ certainly exist in various noncyclic abelian groups. The general question of when noncyclic abelian groups admit nontrivial near-factorizations of a specified index $\lambda \geq 1$ is a very interesting open problem.


\begin{thebibliography}{10}
  
  
\bibitem{BHS}
G.~Bacs\'{o}, L.~H\'{e}thelyi and P.~Sziklai,  New near-factorizations of
  finite groups,  \textsl{Studia Sci. Math. Hungar.} \textbf{45} (2008),
  493--510.
  
\bibitem{DB}  N.G. De Bruijn.
On number systems
\emph{Nieuw Arch. Wisk.} {\bf 3} (1956), 15--17.
  
\bibitem{CGHK} D. de Caen, D.A. Gregory, I.G. Hughes and D.L. Kreher.
Near-factors of finite groups.
\emph{Ars Combin.} {\bf 29} (1990), 53--63.

\bibitem{Grin}
C.M. Grinstead. On circular critical graphs.
\emph{Discr. Math.} {\bf 51} (1984), 11--24.

\bibitem{HR}
C.J.\ Hillar and D.L.\ Rhea.
Automorphisms of finite abelian groups.
\emph{Amer. Math. Monthly} {\bf 114} (2007), 917--923.

\bibitem{HP}
S. Huczynska and M.B. Paterson. 
Existence and non-existence results for strong external difference families.
\emph{Discr. Math.} \textbf{341} (2018), 87--95.


\bibitem{KPS}
D.L.\ Kreher, M.B. Paterson and D.R. Stinson.
Strong external difference families and classification of $\alpha$-valuations.
Preprint.
\url{https://arxiv.org/abs/2406.09075}

\bibitem{PS}
M.B. Paterson and D.R.\ Stinson. 
Circular external difference families, graceful labellings and cyclotomy.
\emph{Discr. Math.} \textbf{347} (2024), article 114103, 15 pp.

\bibitem{Pech03}
A. P\^{e}cher. Partitionable graphs arising from
near-factorizations of finite groups.
\emph{Discr. Math.} {\bf 269} (2003) 191--218.

\bibitem{Pech}
A. P\^{e}cher. Cayley partitionable graphs and near-factorizations of finite groups.
\emph{Discr. Math.} {\bf 276} (2004) 295--311.


\bibitem{SS}
T. Sakuma and H. Shinohara.
Krasner near-factorizations
and 1-overlapped factorizations.
In ``The Seventh European Conference on Combinatorics, Graph Theory and Applications'', {\it EuroComb 2013}, Pisa, 2013, pp.\ 391--395.

\bibitem{Shin}
H. Shinohara.
Thin Lehman matrices arising from finite groups.
\emph{Linear Algebra and its Applications} \textbf{436} (2012), 850--857.


\bibitem{VS}
S. Veitch and D.R.\ Stinson.
Unconditionally secure non-malleable secret sharing and circular external difference families.
\emph{Designs, Codes, Cryptogr.} \textbf{92} (2024), 941--956.

\bibitem{WYF}
H. Wu, J. Yang and  K. Feng.
Circular external difference families: construction and nonexistence.
\emph{Designs, Codes, Cryptogr.} \textbf{92} (2024), 3377--3390.


\end{thebibliography}
\end{document}